\documentclass{amsart}
\usepackage{}
\usepackage{amsfonts}
\usepackage{amsfonts,amssymb,amsmath,amsthm,bbm}
\usepackage{mathrsfs}
\usepackage{graphicx}
\usepackage{url}
\usepackage{enumerate}

\newtheorem{theorem}{Theorem}[section]
\newtheorem{lemma}[theorem]{Lemma}
\newtheorem{proposition}[theorem]{Proposition}

\theoremstyle{definition}
\newtheorem{definition}[theorem]{Definition}
\newtheorem{remark}{Remark}
\newtheorem{example}{Example}
\numberwithin{equation}{section}

\DeclareMathOperator{\vol}{vol}

\DeclareMathOperator{\Pf}{Pf}\DeclareMathOperator{\V}{V}

\author{Wei Zhao}
\address{
Department of Mathematics\\
East China University of Science and Technology\\
Shanghai, China}
\email{szhao\underline{ }wei@yahoo.com}

\begin{document}

\title[]{A Gauss-Bonnet-Chern theorem for Finsler vector bundles}

\begin{abstract}In this paper, we give a simple proof of the Gauss-Bonnet-Chern theorem for a real oriented Finsler vector bundle with rank equal to the dimension of the base manifold. As an application, a Gauss-Bonnet-Chern formula for any metric-compatible connection is established on Finsler manifolds.
\end{abstract}
\maketitle

\section{Introduction} 

Fifty years ago,
S. S. Chern \cite{Ch1,Ch2} gave an intrinsic proof of the Gauss-Bonnet-Chern (GBC) theorem for all oriented closed $n$-dimensional Riemannian manifolds $(M,g)$, that is,
\[
\int_M\mathbf{\Omega}=\chi(M),\tag{1.1}\label{1.11}
\]
where
\begin{align*}
\mathbf{\Omega}&=\left\{
\begin{array}{lll}
&\frac{(-1)^{p}}{2^{2p}\pi^pp!}\epsilon_{i_1\ldots i_{2p}}\Omega_{i_1}^{i_2}\wedge\cdots\wedge \Omega_{i_{2p-1}}^{i_{2p}},&n=2p,\tag{1.2}\label{1.2}\\
\\
&0,&n=2p+1.
\end{array}
\right.
\end{align*}
and $(\Omega^i_j)$ is the local curvature form of the Levi-Civita connection.
The GBC theorem plays an important role in differential geometry, leading to the development in
several areas such as the theory of characteristic classes and
index theory \cite{BGV,MQ}. Recently, the GBC theorem has been generalized to Riemannian bundles \cite{B,N}, which reveals an intrinsically
beautiful fact: the integral of the geometric Euler class is exactly the Euler characteristic for any oriented Riemannian vector bundle with rank equal to the dimension of the base manifold (cf. \cite[Theorem 1]{B}). The rank requirement here is natural and necessary, which makes it possible to integrate the Euler class over the underlying manifold.

Finsler geometry is just Riemannian geometry without quadratic restriction. It is natural to ask whether an analogue of (1.1) still holds for Finsler manifolds. The purpose of this paper is to study this problem.
Let $(M,F)$ be an $n$-dimensional Finsler manifold. Denote by $\pi:SM\rightarrow M$ the projective sphere bundle and $\pi^*TM$ the pull-back tangent bundle. $F$ induces naturally  a Riemannian metric on $\pi^*TM$.
There are various connections of $F$, but none of them is both "torsion-free" and "metric-compatible". For example, the Cartan connection is metric-compatible but not torsion-free  while the Chern connection is torsion-free but not metric-compatible. Refer to \cite{AIM,BCS} for other interesting connections.

In the 1950s,
Lichnerowicz \cite{Li} and Busemann \cite{Bu} made some efforts to generalize (\ref{1.11}) to Finsler manifolds. Their work show that there is no simple formula as (\ref{1.11}) valid for general Finsler manifolds.
Fifty years later, Bao-Chern \cite{BC} reconsidered this problem and established the GBC theorem for Finselr manifolds with $\V(x)=\text{constant}$ by the Chern connection, which is also derived by Shen \cite{Sh1} from the Cartan connection in the same year. Here, $\V(x)$ is the Riemannian volume of $S_xM$ induced by $F$ (cf. \cite{BS}). Five years later, Lackey \cite{L} used a nice trick to deal with $\V(x)$ and generalized the result of Bao-Chern \cite{BC} to general Finsler manifolds. In fact, Lackey established a GBC theorem for any torsion-free connection.

All the methods used in \cite{BC,L,Li,Sh1} are inspired by Chern's original idea, that is, using two polynomials $\Phi_k$, $\Psi_k$ to obtain the transgression
\[
\mathbf{\Omega}^D+\mathfrak{F}=d\Pi,
\]
where $D$ is a connection of $F$, $\mathfrak{F}$ is a form on $SM$, $\mathbf{\Omega}^D$ is defined as (\ref{1.2}) by $D$, and $\Pi$, $\Phi_k$ and $\Psi_k$ are defined as in \cite{Ch1,Ch2}.
It should be noticed that $\Pi$ induced by a torsion-free connection has a close relationship with $\V(x)$, i.e.,
\[
\Pi|_{S_xM}=\frac{d\nu_x}{\vol(\mathbb{S}^{n-1})},\tag{1.3}\label{1.4}
\]
where $d\nu_x$ is the volume form of $S_xM$, that is, $\V(x)=\int_{S_xM}d\nu_x$.
(\ref{1.4})  allows us to use the Poincar\'e-Hopf theorem and establish the GBC theorem \cite{BC,L} (for any torsion-free connection)
\begin{align*}
\int_{M}[X]^*\left(\frac{\mathbf{\Omega}^D+\mathfrak{F}}{\V(x)}\right)=
\frac{\chi(M)}{\vol(\mathbb{S}^{n-1})},
\end{align*}
where $[X]:M\rightarrow SM$ is an arbitrary section with isolated singularities. The additional item $\mathfrak{F}$ does not vanish simply because $D$ is not metric-compatible.
However, in the metric-compatible case, (\ref{1.4}) is no longer true, although $\mathfrak{F}$ always vanishes. In fact, it seems that the metric-compatible structure equation gives little information about $\Pi$.

The purpose of this paper is to give a short proof of the GBC theorem for any metric-compatible connection and for any oriented Finsler bundle with rank equal to the dimension of the base manifold.
In order to introduce our main results, we shall introduce some notions and basic facts of Finsler bundles. See Section 2 below for more details.

Given a rank $n$ oriented Finsler bundle $(\mathscr{E},F)$
over a $n$-dimensional
closed oriented manifold $M$. Let $\pi:S\mathscr{E}\rightarrow M$ be the the projective sphere bundle and let $\pi^*\mathscr{E}$ be the pull-back bundle. $F$ induces an Riemannian metric on $\pi^*\mathscr{E}$. For each metric-compatible connection $D$ on $\pi^*\mathscr{E}$,
define the Pfaffian $\mathbf{\Omega}^D$ as (\ref{1.2}).
Then we have the following:
\begin{theorem}\label{Th1}
Let $\mathscr{E}$ be an oriented Finsler vector bundle of rank $n$ over an $n$-dimensional
closed oriented manifold $M$. Given any metric-compatible connection $D$, for any smooth section $X$ with isolated zeros on $\mathscr{E}$, we have
\begin{align*}
\int_{M}[X]^*\left(\frac{\mathbf{\Omega}^D+\mathfrak{E}}{\V(x)}\right)=
\frac{\chi(\mathscr{E})}{\vol(\mathbb{S}^{n-1})},
\end{align*}
where $[X]$ is the section of $S\mathscr{E}$ induced by $X$, $\V(x)$ is the Riemannian volume of $\pi^{-1}(x)$, $\chi(\mathscr{E})$ is the Euler characteristic of $\mathscr{E}$ and $\mathfrak{E}$ is an $n$-form on $S\mathscr{E}$.
\end{theorem}

It should be remarked that all the $({\mathbf{\Omega}^D+\mathfrak{E}})/{\V(x)}$s of metric-compatible connections are in the same DeRham cohomology
class, which can be viewed as the modified geometric Euler class of $\pi^*\mathscr{E}$. 
The correction term
$\mathfrak{E}$ is an exact form if $\V(x)$ is constant. However, even in this case, $\int_M [X]^*({\mathfrak{E}}/{\V(x)})$ does not vanish (e.g., \cite{Sh1}).
See Section 3 below for the precise formula of $\mathfrak{E}$. Theorem \ref{Th1} reduces to the GBC theorem for Riemannian bundles \cite[Theorem 1]{B}.
\begin{theorem}[\cite{B}]\label{Rvb}
Let $\mathscr{F}$ be a rank $n$ oriented Riemannian vector bundle over an $n$-dimensional
closed oriented manifold $M$. Given any metric-compatible connection $\mathcal {D}$, we have
\[
\int_M \mathbf{\Omega}^\mathcal {D} =\chi(\mathscr{\mathscr{F}}).
\]
\end{theorem}

Recall that $\chi(M)$ is exactly $\chi(TM)$. Then we derive the following GBC theorem for Finsler manifolds from Theorem \ref{Th1} directly.
\begin{theorem}\label{Th2}
Let $(M,F)$ be an $n$-dimensional closed oriented Finsler manifold. Given any metric-compatible connection $D$, for any vector field $X$ with isolated zeros, we have
\begin{align*}
\int_{M}[X]^*\left(\frac{\mathbf{\Omega}^D+\mathfrak{E}}{\V(x)}\right)=
\frac{\chi(M)}{\vol(\mathbb{S}^{n-1})},
\end{align*}
where $[X]$ is the section of $SM$ induced by $X$, $\V(x)$ is the Riemannian volume of $S_xM$, and $\mathfrak{E}$ is an $n$-form on $SM$.
\end{theorem}

Theorem \ref{Th2} implies the classical GBC theorem in
the special case when $M$ is a Riemannian manifold endowed with any metric-compatible connection \cite{BGV,MQ}.
\begin{theorem}[\cite{BGV,MQ}]\label{cr}
Let $(M,g)$ be an closed oriented Riemannian manifold. Given any metric-compatible connection $\mathcal {D}$, we have
\[
\int_M \mathbf{\Omega}^\mathcal {D} =\chi(M).
\]
\end{theorem}

For an oriented Riemannian bundle $\mathscr{F}$, the cohomology class $[\mathbf{\Omega}^\mathcal {D}]$ is the geometric Euler class of $\mathscr{F}$.
Bell obtains Theorem \ref{Rvb} above by revealing an important fact, i.e.,
 the geometric Euler class always coincides with the topological Euler class for all oriented Riemannian bundles \cite[Theorem 4]{B}. However, it seems impossible to extend this result to general Finsler bundles. See Remark \ref{last} below for more details.
 We know that Mathai-Quillen \cite{MQ} give a proof of the Theorem \ref{cr} by showing that the pullback of the Thom class via the zero section is the Euler class of the base manifold. But their method cannot be applied to the Finsler setting either (see Remark \ref{mqm} below).
The key idea in the proof of
Theorem \ref{Th1} is to modify a given metric-compatible connection to another new metric-compatible connection with certain special properties.

It should be noticed that
 Shen \cite{Sh2} has also established a GBC theorem for any metric-compatible connection on Finsler manifolds. But his formula and method are different from ours here. Refer to \cite{Sh2} for more details.

\section{Preliminaries}
In this paper,
the rules that govern our index gymnastics are as follows: low case Latin indices run from $1$ to $n$; capital Latin indices run from $1$ to $m$; and low case Greek indices run from $1$ to $n-1$.

A Finsler vector bundle $(\mathscr{E},F,M)$ is a real vector bundle $\mathscr{E}$ of rank $n$ over a
$m$-dimensional manifold $M$, equipped with a Finsler metric $F$. The Finsler metric $F$ is a nonnegative function on $\mathscr{E}$ satisfying the following three conditions:

(1) $F$ is smooth on  the {slit bundle} $\mathscr{E}\backslash0$;

(2) $F$ is positively homogeneous, i.e., $F(\lambda y)=\lambda F(y)$, for any $\lambda>0$ and $y\in \mathscr{E}$;

(3) The Hessian $\frac{1}{2}[F^2]_{y^iy^j}(x,y)$ is positive definite, where $F(x,y):=F(y^is_i|_x)$ and $\{s_i\}$ is a local frame field of $\mathscr{E}$.

It should be noted that the first condition is natural and important. In fact, a Finsler metric cannot be smooth at the zero-section unless it is Riemannian. For this reason, most of the geometric quantities of a Finsler bundle cannot be defined at the zero-section.

Let $\pi:S\mathscr{E}\rightarrow M$ be the the projective sphere bundle and let $\pi^*\mathscr{E}$ be the pull-back bundle.
 For each $(x,[y])\in S\mathscr{E}$, the tautological section $\ell$ of $\pi^*\mathscr{E}$ is defined by
\[
\ell_{(x,[y])}=\frac{y^i}{F(y)}{\mathfrak{s}_i},
\]
where $\mathfrak{s}_i:=(x,[y],s_i|_x)$, $i=1,\cdots, n$, denote the local frame of $\pi^*\mathscr{E}$.

The Finsler metric $F$ induces naturally a Riemannian metric $g$ and the Cartan tensor $A$ on $\pi^*\mathscr{E}$. Let $(x^A,y^j)=[y^js_j]|_x$ be a local homogeneous coordinate system of $S\mathscr{E}$. Thus, $g=g_{ij}\,\mathfrak{t}^i\otimes \mathfrak{t}^j$ and $A=A_{ijk}\,\mathfrak{t}^i\otimes \mathfrak{t}^j\otimes \mathfrak{t}^k$, where
\[
g_{ij}(x,[y]):=\frac12\frac{\partial^2
F^2(x,y)}{\partial y^i\partial
y^j},\ \ A_{ijk}(x,[y]):=\frac{F}{4}\frac{\partial^3 F^2(x,y)}{\partial
y^i\partial y^j\partial y^k},
\]
and $\{\mathfrak{t}^j\}$ is the dual frame field of $\{\mathfrak{s}_j\}$.
It is easy to check that $F$ is Riemannian if and only if $A=0$. Set $(g^{ij}):=(g_{kl})^{-1}$ and $A^j_{ik}:=g^{jl}A_{lik}$.

Given an Ehresmann connection $\Theta\in \mathscr{A}^1(\mathscr{E},V\mathscr{E})$ on $\mathscr{E}$ (cf. \cite[Definition 1.10]{BGV}), we have a horizontal decomposition
\[
T(\mathscr{E}\backslash 0)=H(\mathscr{E}\backslash 0)\oplus V(\mathscr{E}\backslash 0),
\]
where $H(\mathscr{E}\backslash 0):=\text{ker}\Theta|_{\mathscr{E}\backslash 0}$ and $V(\mathscr{E}\backslash 0)$ denote the horizontal bundle and the vertical bundle of $\mathscr{E}\backslash 0$, respectively. Let $(x^A,y^j)=y^js_j|_{x}$ be a local coordinate system of $\mathscr{E}$. Set $N^j_A:=\langle\Theta(\frac{\partial}{\partial x^A}),dy^j \rangle$. Then we obtain two frame fields $\{\frac{\delta}{\delta x^A}\}_{A=1}^m$ and $\{ \frac{\delta}{\delta y^i}\}_{i=1}^n$ of $H (\mathscr{E}\backslash 0)$ and $V (\mathscr{E}\backslash 0)$ respectively, where
\begin{align*}
\left\{
\begin{array}{lll}
&\frac{\delta}{\delta x^A}:=\frac{\partial}{\partial x^A}-N^k_A\frac{\partial}{\partial y^k},\\
&\frac{\delta}{\delta y^i}:=F\frac{\partial}{\partial y^i}.
\end{array}
\right.
\end{align*}
It is easy to check that
\begin{align*}
\left\{
\begin{array}{lll}
&dx^A,\\
&{\delta y^i}:=\frac{dy^i+N^i_Adx^A}{F},
\end{array}
\right.
\end{align*}
is the dual frame field of $\{\frac{\delta}{\delta x^A}, \frac{\delta}{\delta y^i}\}$.
In particular, if we view $(x^A,y^j)$ as a local homogeneous coordinate system of $S\mathscr{E}$, then $\{\frac{\delta}{\delta x^A}, \frac{\delta}{\delta y^i}\}$ (resp. $\{dx^A,{\delta y^i}\}$) is also a local frame field of $T(S\mathscr{E})$ (resp. $T^*(S\mathscr{E})$). Hence, we obtain two horizontal decompositions
\[
T(S\mathscr{E})=H(S\mathscr{E})\oplus V(S\mathscr{E}),\ T^*(S\mathscr{E})=H^*(S\mathscr{E})\oplus V^*(S\mathscr{E}).
\]
Then the differential operator $d:C^\infty(S\mathscr{E})\rightarrow \mathscr{A}^1(S\mathscr{E})$ is
decomposed into $d=d^H+d^V$, where
\[
d^Hf:=\frac{\delta f}{\delta x^A}d x^A, \ d^Vf:=\frac{\delta f}{\delta y^i}\delta y^i, \ \forall f\in C^\infty(S\mathscr{E}).
\]
In the following, we assume that the Ehresmann connection $\Theta$ has been chosen. 

\begin{example}
Let $\mathscr{T}$ denote the tangent bundle of a Finsler manifold. Then the canonical horizontal decomposition of $T(\mathscr{T}\backslash 0)$ can be derived from the Bott connection (see \cite[p.\,35-38]{An} or \cite[(2.3.2a)]{BCS}). In particular, $\{\frac{\delta}{\delta x^i}\}$ and $\{\frac{\delta}{\delta y^i}\}$ have the same behavior under transformations induced by coordinate changes. However, this is no longer true for a general Finsler bundle $(\mathscr{E},F,M)$ even if $\text{rank}(\mathscr{E})=\text{dim}(M)$. For this reason, all the known connections in Finsler manifolds (e.g., the Cartan connection and the Chern connection) cannot be generalized to general Finsler bundles.
\end{example}

\section{Modification of a metric-compatible connection}
This section is devoted to investigating the connections of a Finsler bundle. According to \cite{BCS,K}, an operator $D$ is called a {\it connection of a Finlser bundle} $(\mathscr{E},F)$, if it is a connection on $\pi^*\mathscr{E}$ defined by $F$.
And a connection $D$ of $(\mathscr{E},F)$ is said to be {\it metric-compatible} if it is compatible with the Riemannian metric $g$ on $\pi^*\mathscr{E}$ induced by $F$.
In particular, if $\mathscr{F}$ is a Riemannian bundle and $\mathcal {D}$ is a metric-compatible connection on $\mathscr{F}$, then $\pi^*\mathcal {D}$ is a metric-compatible connection of the {\it Finsler bundle} $\mathscr{F}$.

Let $(\mathscr{E},F,M)$ be a Finsler bundle as in Section 2.
Given a connection $D$ of $\mathscr{E}$, let $(\theta_i^j)$ denote the connection $1$-form of $D$ with respect to the local frame $\{\mathfrak{s}_i\}$, i.e., $D \mathfrak{s}_i=\theta_i^j \otimes \mathfrak{s}_j$. Since $\theta_i^j\in \mathscr{A}^1(S\mathscr{E})$, by the horizontal decomposition of $T^*(S\mathscr{E})$, we decompose $\theta^i_j$ into
\[
\theta^i_j=\gamma^i_{jA}dx^A+\varrho^{i}_{jk}\delta y^k,
\]
 where $(x^A,y^j)$ is a local homogeneous coordinate system of $S\mathscr{E}$.
It is not hard to see that the horizontal component $(\gamma^i_{kA}dx^A)$ defines a connection on $\pi^*\mathscr{E}$, while the vertical component
$\varrho^j_{ik}\mathfrak{t}^i\otimes \mathfrak{t}^k \otimes \mathfrak{s}_j$
is a smooth section of $\pi^*\mathscr{E}^*\otimes\pi^*\mathscr{E}^*\otimes\pi^*\mathscr{E}$. Inspired by this observation and \cite{BS}, we introduce the modification of a connection.
\begin{definition}\label{d1}Let $D$ be a connection of $\mathscr{E}$. The modified connection $\nabla$ of $D$ is also a connection of $\mathscr{E}$, which is defined by
\[
\nabla \mathfrak{s}_i:=\left(\gamma^j_{iA}dx^A+2A^j_{ik}{\delta y^k}\right)\otimes \mathfrak{s}_j,\tag{3.1}\label{2.2}
\]
where $(\gamma^j_{iA}dx^A)$ is the horizontal component of the connection $1$-form of $D$ with respect to $\{\mathfrak{s}_i\}$.
\end{definition}

Denote by $D^H$ the connection defined by the horizontal component of the connection $1$-form of $D$. In this paper, we say that $D$ is {\it partially metric-compatible}, if $D^H$ satisfies
\[
d^H g(X,Y)=g(D^HX,Y)+g(X,D^HY),\  \forall \,X, Y\in \Gamma(\pi^*\mathscr{E}).
\]
Clearly, a metric-compatible connection is always partially metric-compatible. But a partially metric-compatible connection is not metric-compatible unless the Finsler metric is Riemannian.
However, we have the following proposition.
\begin{proposition}
Let $D$ and $\nabla$ be as in Definition \ref{d1}. Thus, $\nabla$ is metric-compatible if and only if $D$ is partially metric-compatible. Hence, the modified connection of a metric-compatible connection is still metric-compatible.
\end{proposition}
\begin{proof}
It is easy to see that
\[
d g_{ij}=d^Hg_{ij}+d^Vg_{ij}=d^Hg_{ij}+4A_{ijk}\delta y^k,
\]
where $g_{ij}=g(\mathfrak{s}_i,\mathfrak{s}_i)$. Hence,
 $\nabla$ is compatible with $g$ if and only if
$\delta g_{ij}/\delta x^A=g_{ik}\gamma^{k}_{jA}+g_{kj}\gamma^{k}_{iA}$,
that is, $D$ is partially metric-compatible.
\end{proof}

\begin{example}Let $\mathscr{T}$ be the tangent bundle of a Finsler manifold.
It follows from \cite[p.\,39]{BCS} that the modified connection of the Berwald connection is the Hashigchi connection, and the modified connection of the Chern connection is the Cartan connection. In particular, the Cartan connection is metric-compatible while the Chern connection is partially metric-compatible.
\end{example}

We will see below that the modified connection of a metric-compatible connection always satisfies some interesting properties, which allow us to establish a GBC theorem for the original connection.

In the rest of the paper, we assume that  $\{e_i\}_{i=1}^n$ is a local $g$-orthonormal field for $\pi^*\mathscr{E}$ with $e_n=\ell$. Let $\nabla$ be the modified connection of a metric-compatible connection. Denote by $\varpi^i_j$ (resp. $\Omega^i_j$) the connection (resp. curvature) form of $\nabla$ with respect to $\{e_i\}$, i.e.,
\[
\nabla e_i=:\varpi_i^j\otimes e_j,\ \Omega_i^j:=d\varpi_i^j-\varpi_i^k\wedge \varpi_k^j.
\]
First we have the following proposition.

\begin{proposition}\label{sp}Let $\nabla$ be the modified connection of a metric-compatible connection.
Then for each $x\in M$, we have
\[
d\nu_x=i^*_x(\varpi^n_1\wedge\cdots\wedge\varpi_{n-1}^n),
\]
where $d\nu_x$ is the Riemannian volume form of $S_x\mathscr{E}:=\pi^{-1}(x)$ induced by $F$, and $i_x:S_x\mathscr{E}\hookrightarrow S\mathscr{E}$ is the injective map.
\end{proposition}
\begin{proof}Let $(x^A,y^i)$ denote a local homogeneous coordinate system of $S\mathscr{E}$. Then
\[
d\nu_x(y)={\sqrt{\det g_{ij}(x,y)}}\overset{n}{\sum_{i=1}}{(-1)^{i-1}}\,\frac{y^{i}}{F}d\left(\frac{y^1}{F}\right)\wedge \cdots \wedge \widehat{d\left(\frac{{y}^i}{F}\right)}\wedge \cdots \wedge d\left(\frac{y^n}{F}\right).
\]

Set $e_i=: B_i^k \mathfrak{s}_k$. It is easy to see that $B^k_n=y^k/F(y)$ and $\det (B^{-1})=\sqrt{\det g_{ij}(x,y)}$. Thus,
$\nabla e_n=\varpi_n^j\otimes e_j$ together with (\ref{2.2}) and $y^kA_{kij}=0$ yields
\[
i_x^*(\varpi_n^k B^i_k)=d\left(\frac{y^i}{F(y)}\right),
\]
which implies that
\[
i^*_x(\varpi^1_n \wedge \cdots \wedge \varpi^{n-1}_n)=(B^{-1})^1_{i_1}\cdots (B^{-1})^{n-1}_{i_{n-1}}d\left(\frac{y^{i_1}}{F}\right)\wedge \cdots \wedge  d\left(\frac{y^{i_{n-1}}}{F}\right).
\]
Let $\{\theta^i\}$ denote the dual frame field of $\{e_i\}$. By $e_n$ contracting $\theta^1\wedge \cdots \wedge \theta^n$, one has
\[
(B^{-1})^1_{i_1}\cdots (B^{-1})^{n-1}_{i_{n-1}}=(-1)^{n-1}\overset{n}{\sum_{i=1}}{(-1)^{i-1}}\det(B^{-1})\frac{y^i}{F}\frac{\delta^{1\cdots \hat{i}\cdots n}_{i_1\cdots i_{n-1}}}{(n-1)!}.
\]
\end{proof}

Before continuing to investigate the modified connection of a metric-compatible connection, we introduce some notions first.
Set $\mathscr{A}^{i,j}:=\Gamma(S\mathscr{E},\wedge^iT^*S\mathscr{E}\otimes \wedge^j \pi^*\mathscr{E})$ and $\mathscr{A}:=\sum_{i,j}\mathscr{A}^{i,j}$. Then $\mathscr{A}$ is a bigraded algebra (cf. \cite{MQ}), that is, for $a\otimes b\in \mathscr{A}^{i,j}$ and $c\otimes d\in \mathscr{A}^{k,l}$, the product of $a\otimes b$ and $c\otimes d$ is defined by
\[
(a\otimes b)\cdot(c\otimes d)=(-1)^{jk}(a\wedge c)\otimes (b\wedge d).
\]
In the rest of this section, we assume that $(\mathscr{E},F)$ is an oriented Finsler bundle.
Thus, $\pi^*\mathscr{E}$ is an oriented bundle and therefore, we can induce the Berezin integral $\mathscr{B}$ (cf. \cite{BGV}) to $\mathscr{A}$ by
\[
\mathscr{B}:a\otimes \eta\in \mathscr{A}\mapsto a(\mathscr{B}\eta)\in \sum_{i}\mathscr{A}^{i,0},
\]
with
\[
\mathscr{B}(e_I)=\left\{
\begin{array}{lll}
&\epsilon_I,&|I|=n,\\
&0,&\text{otherwise}.
\end{array}
\right.
\]
We always identify $\mathfrak{so}(\pi^*\mathscr{E})$ with $\wedge^2 \pi^*\mathscr{E}$ by the map
\[
B\in \mathfrak{so}(\pi^*\mathscr{E})\mapsto\frac12\underset{i,j}{\sum}g(Be_i,e_j)e_i\wedge e_j.
\]
In particular, the curvature of an arbitrary metric-compatible connection is viewed as an element in $\mathscr{A}^{2,2}$.

The arguments similar  to those in \cite[p.52-55]{BCS} show the following lemma.
\begin{lemma}\label{newone}Let $\nabla$ be the modified connection of a metric-compatible connection. Set
\[
U_t:=\mathscr{B}({e^{-\Theta_t}}):= \mathscr{B}\left(e^{-(\frac{t^2}2+\mathbbm{i}t\nabla \ell+\Omega)}\right)=\mathscr{B}\left(e^{-\frac{t^2}2}\overset{\infty}{\underset{k=0}{\sum}}\frac{(-1)^k}{k!}(\mathbbm{i}t\nabla \ell+\Omega)^k\right),
\]
where $\Omega$ is the curvature of $\nabla$ and $\mathbbm{i}:=\sqrt{-1}$. Thus, $U_t$ is a closed $n$-form on $S\mathscr{E}$ and
\[
\frac{d}{dt}U_t=-\mathbbm{i} d\left[\mathscr{B}\left(\ell\cdot {e^{-\Theta_t}}\right)\right].
\]
\end{lemma}
\begin{remark}\label{mqm} Mathai-Quillen's proof of the formula (1.1) was carried out by constructing a Thom form on $TM$, which pulled back by the zero-section is exactly the (geometric) Euler form. Although $U_1$ is similar to the Mathai-Quillen's Thom form restricted to $S\mathscr{E}$, the argument in \cite{MQ} cannot be applied to Finsler bundles, since the connection $1$-form of any connection cannot be extended to the zero-section for a general Finsler metric.
\end{remark}

According to \cite[Definition 1.35]{BGV}, $U_0=\mathscr{B}({e^{-\Omega}})$ is exactly the Pfaffian $\Pf(-\Omega)$.
Lemma \ref{newone} then yields the following result.
\begin{proposition}\label{newe}Let $\nabla$ be the modified connection of a metric-compatible connection. Then
\[
\Pf(-\Omega)=(-1)^{n-1}d\left(\overset{\left[\frac{n-1}{2}\right]}{\underset{k=0}{\sum}}\frac{(-1)^k2^{\frac{n}{2}}\Phi_k}{k!(n-1-2k)!2^{2k+1}}\Gamma\left(\frac{n-2k}{2}\right)\right),
\]
where $\Omega$ is the curvature of $\nabla$, $\Gamma(s)$ is the gamma function and
\[
\Phi_k:=\sum\epsilon_{\alpha_1\ldots \alpha_{n-1}}\Omega_{\alpha_1}^{\alpha_2}\wedge\cdots\wedge\Omega_{\alpha_{2k-1}}^{\alpha_{2k}}\wedge \varpi_{\alpha_{2k+1}}^n\wedge\cdots\wedge \varpi_{\alpha_{n-1}}^n.
\]
\end{proposition}
\begin{proof} Denote by $\Xi$ the component of $e^{-(\mathbbm{i}t\nabla\ell+\Omega)}$ in $\mathscr{A}^{n-1,n-1}$. It is not hard to see that
\begin{align*}
\Xi&=\overset{\infty}{\underset{k=0}{\sum}}\frac{(-1)^k}{k!}\left(\underset{\{s:(k-s)+2s=n-1,\,0\leq s \leq k\}}{\sum}
\left(
\begin{array}{lll}
k\\
s
\end{array}
\right)
(\mathbbm{i}t\nabla\ell)^{k-s}\cdot\Omega^s\right)\\
&=\overset{n-1}{\underset{k=\left[\frac{n}{2}\right]}{\sum}}\frac{(-1)^k}{(n-1-k)!(2k-(n-1))!}(\mathbbm{i}t\nabla\ell)^{2k-(n-1)}\cdot\Omega^{n-1-k}\\
&=(-\mathbbm{i})^{n-1}\overset{\left[\frac{n-1}{2}\right]}{\underset{k=0}{\sum}}\frac{(t\nabla\ell)^{n-1-2k}\cdot\Omega^{k}}{k!(n-1-2k)!}.
\end{align*}
Hence, we obtian
\[
\ell\cdot\Xi=\frac{(-1)^{n-1}}{\epsilon(n+1)}\overset{\left[\frac{n-1}{2}\right]}{\underset{k=0}{\sum}}\frac{t^{n-1-2k}(-1)^k}{k!(n-1-2k)!2^k}\Phi_k\otimes e_1\wedge\cdots\wedge e_n,\tag{3.2}\label{3.2}
\]
where
\begin{align*}
\epsilon(n):=\left\{
\begin{array}{lll}
&1,\ n=2p,\\
&\mathbbm{i},\ n=2p+1.
\end{array}
\right.
\end{align*}

Since $U_t=e^{-t^2/2}\mathscr{B}(e^{-(it\nabla\ell+\Omega)})\rightarrow 0$ (as $t\rightarrow\infty$), Lemma \ref{newone} together with (\ref{3.2}) then yields
\begin{align*}
\Pf(-\Omega)&=U_0=\mathbbm{i} d\left[\int^\infty_0\mathscr{B}(\ell\cdot{e^{-\Theta_t}})dt\right]=\mathbbm{i}d \left[\int^\infty_0 e^{-t^2}\mathscr{B}(\ell\cdot\Xi) dt\right]\\
&=(-1)^{n-1}\epsilon(n)d\left(\overset{\left[\frac{n-1}{2}\right]}{\underset{k=0}{\sum}}\frac{(-1)^k2^{\frac{n}{2}}\Phi_k}{k!(n-1-2k)!2^{2k+1}}\Gamma\left(\frac{n-2k}{2}\right)\right).
\end{align*}

In particular, if $n=2p+1$, then $\Pf(-\Omega)=U_0=0$ and (therefore)
\[
(-1)^{n-1}d\left(\overset{\left[\frac{n-1}{2}\right]}{\underset{k=0}{\sum}}\frac{(-1)^k2^{\frac{n}{2}}\Phi_k}{k!(n-1-2k)!2^{2k+1}}\Gamma\left(\frac{n-2k}{2}\right)\right)
=0=\Pf(-\Omega).
\]
\end{proof}

Let $\nabla$ be the modified connection of a metric-compatible connection $D$.
Define
\[
\mathbf{\Omega}^\nabla:=\frac{1}{(2\pi)^{\frac{n}{2}}}\Pf(-\Omega),\ \Pi:=\frac{(-1)^{n-1}}{\pi^{\frac{n}{2}}}\left(\overset{\left[\frac{n-1}{2}\right]}{\underset{k=0}{\sum}}\frac{(-1)^k\Phi_k}{k!(n-1-2k)!2^{2k+1}}\Gamma\left(\frac{n-2k}{2}\right)\right).
\]
An easy calculation yields that
\begin{align*}
\mathbf{\Omega}^\nabla&=\left\{
\begin{array}{lll}
&\frac{(-1)^{p}}{2^{2p}\pi^pp!}\epsilon_{i_1\ldots i_{2p}}\Omega_{i_1}^{i_2}\wedge\cdots\wedge \Omega_{i_{2p-1}}^{i_{2p}},&n=2p,\\
\\
&0,&n=2p+1.
\end{array}
\right.\\
\Pi&=\left\{
\begin{array}{lll}
&\frac{1}{(2\pi)^p}\overset{p-1}{\underset{k=0}{\sum}}\frac{(-1)^{k+1}}{(2p-2k-1)!!k!2^k}\Phi_k,&n=2p,\\
\\
&\frac{1}{\pi^p2^{2p+1}p!}\overset{p}{\underset{k=0}{\sum}}(-1)^{k}\left(
\begin{array}{lll}
p\\
k
\end{array}
\right)\Phi_k,&n=2p+1.
\end{array}
\right.
\end{align*}
Hence, $\mathbf{\Omega}^\nabla$ and $\Pi$ are of the same form as the ones defined in \cite{BC,Ch1,Ch2,Sh1}.
The proposition above then implies
\[
\mathbf{\Omega}^\nabla=d\Pi. \tag{3.3}\label{3.3}
\]
It is not hard to see that
(\ref{3.3}) is valid not only for all metric-compatible connections but also for all oriented Finsler bundles. Recall that $\nabla$ is the modification of $D$ and therefore, $\nabla$ is also metric-compatible. The Chern-Weil theory then yields
\[
\mathbf{\Omega}^D-\mathbf{\Omega}^\nabla=d\left[\int^1_0\frac{1}{(2\pi)^{\frac{n}{2}}}\mathscr{B}\left(\exp(-\Omega_s)\cdot\frac{\partial{D_s}}{\partial s}\right)ds\right]=:d\Upsilon_0,
\]
where $D_s:=s\nabla+(1-s)D$ and $\Omega_s$ is the curvature of $D_s$ (see Proposition \ref{trass}).

From above, we have
\[
\frac{\mathbf{\Omega}^D+\mathfrak{E}}{\V(x)}=d\left(\frac{\Upsilon_1}{\V(x)}\right),\tag{3.4}\label{3.4}
\]
where $\V(x)$ is the Riemannian volume of $S_x\mathscr{E}$ and
\begin{align*}
&\mathfrak{E}:=-d\Upsilon_0-d\log\V(x)\wedge \Upsilon_1-{d\Upsilon_2},\ \Upsilon_1:=\frac{(-1)^{n-1}}{2\pi^{\frac{n}{2}}}\frac{\Phi_0}{(n-1)!}\Gamma\left(\frac{n}{2}\right),\\ &\Upsilon_2:=\Pi-\Upsilon_1=\frac{(-1)^{n-1}}{\pi^{\frac{n}{2}}}\left(\overset{\left[\frac{n-1}{2}\right]}{\underset{k=1}{\sum}}\frac{(-1)^k\Phi_k}{k!(n-1-2k)!2^{2k+1}}\Gamma\left(\frac{n-2k}{2}\right)\right).
\end{align*}
(\ref{3.4}) implies that all the $({\mathbf{\Omega}^D+\mathfrak{E}})/{\V(x)}$s of metric-compatible connections $D$ are in the same DeRham cohomology
class. Recall that $[\mathbf{\Omega}^D]\in H^n(S\mathscr{E})$ is the geometric Euler class of $\pi^*(S\mathscr{E})$. Thus, $[({\mathbf{\Omega}^D+\mathfrak{E}})/{\V(x)}]\in H^n(S\mathscr{E})$ can be viewed as the modified geometric Euler class of $\pi^*\mathscr{E}$. In particular, if $\V(x)=\text{const.}$, then $\mathfrak{E}$ is an exact $n$-form and $[({\mathbf{\Omega}^D+\mathfrak{E}})/{\V(x)}]$ is the geometric Euler class of $\pi^*\mathscr{E}$ (up to a constant).

\section{Proof of Theorem 1.1}
In this section,
we assume that $(\mathscr{E},F)$ is an oriented Finsler manifold of rank $n$ over an oriented closed $n$-dimensional manifold $M$. The reason why we require
$\text{rank}(\mathscr{E})=\text{dim}(M)$ here is to integrate the pullback of (3.4) over the underlying manifold. Notations are as in Section 3. First we have the following lemma.

\begin{lemma}\label{inter}Let $X\in \Gamma(\mathscr{E})$ be a  smooth section with isolated zeros $\{x_\mathfrak{a}\}_{\mathfrak{a}=1}^k$. Thus, for each $\mathfrak{a}$, there exists a small neighborhood $U_\mathfrak{a}$ of $x_\mathfrak{a}$ such that
\[
\int_{\partial  U_\mathfrak{a}}[X]^*\left(\frac{\Upsilon_1}{\V(x_\mathfrak{a})}\right)=(-1)^{n-1}\frac{\text{locdeg}(X;x_\mathfrak{a})}{\vol(\mathbb{S}^{n-1})},
\]
where $[X]:M\backslash \cup \{x_\mathfrak{a}\}\rightarrow S\mathscr{E} \backslash \cup S_{x_\mathfrak{a}}\mathscr{E}$ is the section induced by $X$, and
$\text{locdeg}(X;x_\mathfrak{a})$ is the local degree of $X$ at $x_\mathfrak{a}$.
\end{lemma}

\begin{proof}Step 1. We shall construct a Finsler metric $\widetilde{F}$ on $M$ such that $(T_{x_\mathfrak{a}}M,\widetilde{F})$ is isometric to $(\mathscr{E}_{x_\mathfrak{a}},F)$, for all isolate zeros $x_\mathfrak{a}$.

 For each $x_\mathfrak{a}$, there exists a local coordinate system $(V_\mathfrak{a}, u_\mathfrak{a}^i)$ consistent with the orientation of $M$ such that $x_\mathfrak{a}\in V_\mathfrak{a}$ and $V_\mathfrak{a}\cap V_\mathfrak{b}=\emptyset$, if $\mathfrak{a}\neq\mathfrak{b}$. Define a Finsler metric $F_\mathfrak{a}$ on $V_\mathfrak{a}$ by
\[
F_\mathfrak{a}\left(y^i\frac{\partial}{\partial u_\mathfrak{a}^i}\right):=F(y^is_{i\mathfrak{a}}),
\]
where $\{s_{i\mathfrak{a}}\}$ is a local frame field consistent with the orientation of $\mathscr{E}$ on $V_\mathfrak{a}$. Let $W_\mathfrak{a}$ be a relatively compact neighborhood of $x_\mathfrak{a}$ such that $x_\mathfrak{a}\in W_\mathfrak{a}\subset \overline{W}_\mathfrak{a}\subset V_\mathfrak{a}$. By \cite[Lemma 1, p.\,26]{Chern0}, one can construct a smooth nonnegative function $\rho_\mathfrak{a}$ on $M$ with $\rho_\mathfrak{a}|_{V_\mathfrak{a}}\neq 0$ and
\begin{align*}
\rho_\mathfrak{a}(x)&=\left\{
\begin{array}{lll}
&1,&x\in \overline{W}_\mathfrak{a},\\
&0,&x\in M\backslash V_\mathfrak{a}.
\end{array}
\right.
\end{align*}
Choose a cut-off function $\rho_0$ on $M$ such that
\begin{align*}
\rho_0&=\left\{
\begin{array}{lll}
&1,&x\in M\backslash(\cup_\mathfrak{a} \,V_\mathfrak{a}),\\
&0,&x\in \cup_\mathfrak{a} W_\mathfrak{a}.
\end{array}
\right.
\end{align*}
Now set $h_s:=\rho_s/(\rho_0+\sum_\mathfrak{a}\rho_\mathfrak{a})$, $s=0,1,\cdots,k$.
Choose an arbitrary Riemannian metric $g$ on $M$ and define a function $\widetilde{F}$ on $TM$ by
\[
\widetilde{F}:=\sum_{\mathfrak{a}} h_\mathfrak{a} F_\mathfrak{a}+h_0 \sqrt{g}.
\]
Proposition \ref{fIN} then yields that $\widetilde{F}$ is a Finsler metric on $M$. Clearly, $(T_{x_\mathfrak{a}}M,\widetilde{F})$ is isometric to $(\mathscr{E}_{x_\mathfrak{a}},F)$ for all $\mathfrak{a}$.

Step 2. By the Finsler metric $\widetilde{F}$, we can use the argument in \cite[Theorem 0.1]{Sh1}.
Choose a small $\epsilon>0$ such that $B^+_{x_\mathfrak{a}}(\epsilon)\subset V_\mathfrak{a}$ for all $\mathfrak{a}$, where $B^+_{x_\mathfrak{a}}(\epsilon)$ is the forward $\epsilon$-ball defined by $\widetilde{F}$.
Define a map $\varphi_\epsilon: S_{x_\mathfrak{a}}\mathscr{E}\rightarrow S\mathscr{E}$ by
\[
\varphi_\epsilon([y])=[X]\circ\kappa([y]),
\]
where $\kappa([y]):=\exp_{x_\mathfrak{a}}\left(\epsilon \mathscr{I}_\mathfrak{a}([y])\right)$ and $\mathscr{I}_\mathfrak{a}:(\mathscr{E}_{x_\mathfrak{a}}, F)\rightarrow (T_{x_\mathfrak{a}}M,\widetilde{F})$ is the isometry constructed as above. $(\exp_{x_\mathfrak{a}})_{*0}=\text{id}$ implies that $\kappa$ is a diffeomorphism between $S_{x_\mathfrak{a}}\mathscr{E}$ and $\partial B^+_{x_\mathfrak{a}}(\epsilon)$. Hence, $\text{deg}(\kappa)=1$ and therefore,
\[
\text{deg}(\varphi_\epsilon)=\text{deg}([X])\circ\text{deg}(\kappa)=\text{deg}([X])=\text{locdeg}(X;x_\alpha).
\]
Recall that $\vol(\mathbb{S}^{n-1})=2\pi^{n/2}/(\Gamma(n/2))$ and $\Phi_0|_{S_{x_\mathfrak{a}}\mathscr{E}}=(n-1)!d\nu_{x_\mathfrak{a}}$ (see Proposition \ref{sp}). Thus,
\begin{align*}
&\int_{\partial B^+_{{x_\mathfrak{a}}}(\epsilon)}[X]^*\left(\frac{\Upsilon_1}{\V({x_\mathfrak{a}})}\right)=\int_{\kappa(S_{x_\mathfrak{a}}\mathscr{E})}[X]^*\left(\frac{\Upsilon_1}{\V({x_\mathfrak{a}})}\right)
=\text{deg}(\varphi_\epsilon)\int_{S_{x_\mathfrak{a}}\mathscr{E}}\frac{\Upsilon_1}{\V({x_\mathfrak{a}})}\\=&(-1)^{n-1}\frac{\text{locdeg}(X;{x_\mathfrak{a}})}{\vol(\mathbb{S}^{n-1})}.
\end{align*}
We are done by choosing $U_\mathfrak{a}:=B^+_{x_\mathfrak{a}}(\epsilon)$ for all $\mathfrak{a}$.\end{proof}

We recall the following generalized Poincar\'e-Hope theorem \cite{BT}.
\begin{theorem}[generalized Poincar\'e-Hope theorem]\label{PH}
Let $\mathscr{E}$ be an oriented rank $n$ vector bundle over a closed oriented $n$-dimensional
manifold $M$. For any smooth section $X\in \Gamma(\mathscr{E})$ with isolated zeros $\{x_\alpha\}$, we have
\[
\chi(\mathscr{E}):=\int_Me(\mathscr{E})=\underset{\mathfrak{a}}{\sum}\text{locdeg}(X;x_\mathfrak{a}),
\]
where $\chi(\mathscr{E})$ (resp. $e(\mathscr{E})$) is the Euler characteristic (resp. the topological Euler class) of $\mathscr{E}$.
\end{theorem}

Now we prove Theorem \ref{Th1}.

\begin{proof}[of Theorem {\rm{1.1}}]
Denote by $\{x_\mathfrak{a}\}$ the isolated zeros of $X$. Let $U_\mathfrak{a}=B^+_{x_\mathfrak{a}}(\epsilon)$ be as in Lemma \ref{inter}. Thus,
(\ref{3.4}) together with Lemma \ref{inter} and Theorem \ref{PH} yields
\begin{align*}
&\int_{M\backslash \cup B^+_{x_\mathfrak{a}}(\epsilon)}[X]^*\left(\frac{\mathbf{\Omega}^D+\mathfrak{E}}{\V(x)}\right)=\int_{M\backslash \cup B^+_{x_\mathfrak{a}}(\epsilon)}[X]^*d\left(\frac{\Upsilon_1}{\V(x)}\right)\\
=&-\underset{\mathfrak{a}}{\sum}\int_{\partial B^+_{x_\mathfrak{a}}(\epsilon)}[X]^*\left(\frac{\Upsilon_1}{\V(x)}\right)= \frac{(-1)^n}{\vol(\mathbb{S}^{n-1})}\underset{\mathfrak{a}}{\sum}\text{locdeg}(X;x_\mathfrak{a})=\frac{(-1)^n}{\vol(\mathbb{S}^{n-1})}\chi(\mathscr{E}).
\end{align*}
Note that $e(\mathscr{E})=0$ if $n=2p+1$ (see \cite[Theorem 8.3.17]{N}). We finish the proof by letting $\epsilon\rightarrow 0^+$.
\end{proof}

\begin{proof}[of Theorem {\rm{\ref{Rvb}}} and Theorem {\rm{\ref{cr}}}]We just prove Theorem \ref{Rvb}. Likewise, one can show Theorem \ref{cr}.
Since $\mathscr{F}$ is a Riemannian bundle, the pull-back connection $\pi^*\mathcal {D}$ is a metric-compatible connection on $\pi^*\mathscr{F}$.
Note that the modified connection of $\pi^*\mathcal {D}$
is exactly itself, which implies that $d\Upsilon_0=0$.
Since $\V(x)=\vol(\mathbb{S}^{n-1})$, $\mathfrak{E}=-d\Upsilon_2$ has no pure-$dy$ part and (therefore) $\int_M[X]^*\mathfrak{E}=0$. Hence, Theorem 1.1 implies
\[
\chi(\mathscr{E})=\int_M[X]^*\mathbf{\Omega}^{\pi^*\mathcal {D}}=\int_M[X]^*\pi^*\mathbf{\Omega}^\mathcal {D}=\int_M\mathbf{\Omega}^\mathcal {D}.
\]
\end{proof}

\begin{remark}\label{last}
In \cite{B}, Bell shows that for any oriented Riemannian bundle $\mathscr{F}$ of even rank over an oriented closed manifold, the geometric Euler class always coincides with the topological Euler class, i.e.,
\[
[\mathbf{\Omega}^\mathcal {D}]=e(\mathscr{F}).\tag{4.1}\label{4.4}
\]
In fact, (\ref{4.4}) holds for any rank (cf. \cite[Theorem 8.3.17]{N}).
Theorem \ref{Rvb} then follows immediately.
(\ref{4.4}) is so beautiful that one might expect to generalize it to the Finsler setting. But it seems impossible.  First, for an oriented Finsler bundle $(\mathscr{E},F,M)$ of rank $n$, the cohomology
class $[\mathbf{\Omega}^D]\in H^n(S\mathscr{E})$ (resp. $[({\mathbf{\Omega}^D+\mathfrak{E}})/{\V(x)}]\in H^n(S\mathscr{E})$) is the geometric Euler class (resp. the modified geometric Euler class) of the pull-back bundle $\pi^*\mathscr{E}$, while $e(\mathscr{E})\in H^n(M)$ is the topological Euler class of the original bundle $\mathscr{E}$; in general, these
classes are not in the same cohomology space. Secondly, $e(\mathscr{E})$ is defined by the pullback of the Thom class via the zero-section, but most of the quantities (especially, the connections and curvatures) of a general Finsler bundle cannot be defined at the zero-section.
Hence, there is no relation between the (modified) geometric Euler class of $\pi^*\mathscr{E}$ and the topological Euler class of $\mathscr{E}$, except for the trivial result
\[
\left[\frac{\mathbf{\Omega}^D+\mathfrak{E}}{\V(x)}\right]=[\mathbf{\Omega}^D]={\pi^*e(\mathscr{E})}=0.
\]
\end{remark}

\section{Appendix}

\begin{proposition}\label{trass}
Let $\nabla$ be the modified connection of a metric-compatible connection $D$. Then
\[
\mathbf{\Omega}^D-\mathbf{\Omega}^\nabla=d\left[\int^1_0\frac{1}{(2\pi)^{\frac{n}{2}}}\mathscr{B}\left(\exp(-\Omega_s)\cdot\frac{\partial{D_s}}{\partial s}\right)ds\right],
\]
where $D_s:=s\nabla+(1-s)D$ and $\Omega_s$ is the curvature of ${D_s}$.
\end{proposition}
\begin{proof}
Clearly, $\{D_s\}$ is a family of metric-compatible connections.
Set ${D_s}=:d+\omega_s$, where $\omega_s\in \mathscr{A}^1(S\mathscr{E})\otimes \mathfrak{so}(\pi^*\mathscr{E})$. Thus,
\[
\frac{\partial{D_s}}{\partial s}=\frac12\sum\frac{\partial(\omega_s)^j_i}{\partial s}\otimes e_i\wedge e_j\in \mathscr{A}^{1,2},\ \frac{\partial}{\partial s}\Omega_s= {D_s}\left(\frac{\partial {D_s}}{\partial s}\right)\in \mathscr{A}^{2,2}.
\]
Since $D_s\Omega_s=0$, we have
\begin{align*}
\frac{\partial}{\partial s}\mathbf{\Omega}^{{D_s}}&=\frac{1}{(2\pi)^{\frac{n}{2}}}\frac{\partial}{\partial s}\mathscr{B}(\exp(-\Omega_s))=\frac{-1}{(2\pi)^{\frac{n}{2}}}\mathscr{B}\left({D_s}\left(\exp(-\Omega_s)\cdot\frac{\partial{D_s}}{\partial s}\right)\right)\\
&=\frac{-1}{(2\pi)^{\frac{n}{2}}}d\mathscr{B}\left(\exp(-\Omega_s)\cdot\frac{\partial{D_s}}{\partial s}\right).
\end{align*}
\end{proof}

\begin{proposition}\label{fIN}
Let $F_i$, $i=1,2$ be two Minkowski norms on $\mathbb{R}^n$. Then $\widetilde{F}:=F_1+F_2$ is still a Minkowski norm.
\end{proposition}
\begin{proof}
It is easy to check the regularity and the positive homogeneity of $\widetilde{F}$. We just show $\widetilde{F}$ is strictly convex. Let $(y^i)$ denote the coordinates in $\mathbb{R}^n$ and let $g_s$, $s=1,2$ (resp. $\tilde{g}$) denote the fundamental tensor of $F_s$ (resp. $\widetilde{F}$). For any $y\neq 0$ and $X=X^i\frac{\partial}{\partial y^i}$, we have
\begin{align*}
&\tilde{g}_{(y)}(X,X)
=\\
&\left[g_{1(y)}\left(\frac{y}{F_1(y)},X\right)+g_{2(y)}\left(\frac{y}{F_1(y)},X\right) \right]^2+\widetilde{F}(y)\left(\frac{\partial F_1}{\partial y^i\partial y^j}X^iX^j+\frac{\partial F_2}{\partial y^i\partial y^j}X^iX^j\right).
\end{align*}
It follows from \cite[(1.2.9)]{BCS} that $\tilde{g}_{(y)}(X,X)\geq 0$ with equality if and only if
\begin{align*}
\left\{
\begin{array}{lll}
&g_{1(y)}\left(\frac{y}{F_1(y)},X\right)=-g_{2(y)}\left(\frac{y}{F_2(y)},X\right), &(*1)\\
\\
&\frac{\partial F_s}{\partial y^i\partial y^j}X^iX^j=0,\ s=1,2. &(*2)
\end{array}
\right.
\end{align*}
(*2) together with \cite[(1.2.7)-(1.2.9)]{BCS} yields that $X=\beta y$, $\beta\in \mathbb{R}$. Then (*1) implies $\beta F_1(y)=-\beta F_2(y)$, i.e., $\beta=0$.
\end{proof}

\end{document}